\documentclass{article}
\usepackage[utf8]{inputenc}
\usepackage{amsmath,amsthm,fullpage,graphicx,amssymb}
\usepackage{listings}
\usepackage{authblk}
\usepackage{url}
\usepackage{pdflscape}
\usepackage{algpseudocode,algorithm}
\usepackage{stmaryrd}
\usepackage{cite}
\usepackage{hyperref}
\newtheorem{theorem}{Theorem}
\newtheorem{lemma}{Lemma}
\newtheorem{corollary}{Corollary}

\newtheorem{conjecture}{Conjecture}

\title{Algorithms for complementary sequences}
\author{Chai Wah Wu}
\affil{Mathematics of Computation\\IBM Research\\IBM T. J. Watson Research Center, Yorktown Heights, NY 10598, USA\thanks{cwwu@us.ibm.com}}

\date{September 9, 2024\\Latest update: September 12, 2025}

\begin{document}

\maketitle

\centerline{\bf Abstract}
\noindent
Finding the $n$-th positive square number is easy, as it is simply $n^2$. But how do we find the complementary sequence, i.e., the $n$-th positive non-square number? For this case there is an explicit formula. However, for general constraints on numbers, a formula is harder to find. In this paper, we study how to compute the $n$-th integer that does (or does not) satisfy a certain condition. In particular, we consider it as a fixed point problem, relate it to the iterative method of Lambek and Moser, study a bisection approach to this problem, and provide novel formulas for various complementary sequences including the non-$k$-gonal numbers, non-$k$-gonal-pyramidal numbers, non-$k$-simplex numbers, non-sum-of-$k$-th-powers, and non-$k$-th-powers. 
For example, we show that the $n$-th non $k$-gonal number is given by $n+\text{round}\left(\sqrt{\frac{2n-2+\left\lfloor\frac{k+1}{4}\right\rfloor}{k-2}}\right)$ and that the $n$-th non-second-hexagonal number is $n+\left\lceil\sqrt{\frac{n}{2}}\right\rceil-1$.

\section{Introduction} \label{sec:intro}
For a positive integer $n\in \mathbb{N}^+$, the $n$-th positive square number is simply $n^2$. Can we also easily find the complementary sequence? In other words, what is the $n$-th positive non-square number? It is quite remarkable that there exists an explicit formula for the $n$-th positive non-square number:
$n+\lfloor \frac{1}{2}+\sqrt{n}\rfloor = n+\lfloor \sqrt{n+\lfloor \sqrt{n}\rfloor}\rfloor$ \cite{Lambek1954,Honsberger1970,Nelson1988,mortici2010}.
This can also be computed as
$n+\lfloor\sqrt{n}\rfloor +1$ if $n-1\geq \lfloor\sqrt{n}\rfloor(\lfloor\sqrt{n}\rfloor+1)$ and as 
$n+\lfloor\sqrt{n}\rfloor$ otherwise. These formulas are well-suited for implementation in a computer algorithm since many computer languages and number theory software packages include functions to compute
$\lfloor\sqrt{n}\rfloor$.
For instance, the \texttt{isqrt} function in Python, Julia, and Maple all perform this calculation.
These formulas have been extended to higher powers as well. In particular, the $n$-th non-$k$-th-power number is given by $n+\lfloor \sqrt[k]{n+\lfloor \sqrt[k]{n}\rfloor}\rfloor$ \cite{Lambek1954,Reis1990,Nyblom2002}.\footnote{The computation of $\lfloor \sqrt[k]{n}\rfloor$ for arbitrary integers $k$ and $n\geq 0$ is readily available in symbolic computer algebra systems and software for number theory. For instance, $\lfloor \sqrt[k]{n}\rfloor$ can be computed with the \texttt{integer\char`_nthroot} function in the \texttt{sympy} Python module which in turn uses the \texttt{mpz\char`_root} function in the multiple precision library \texttt{gmp}. Although these computer operations assume that $n$ is an integer, they can be used to compute $\lfloor \sqrt[k]{n}\rfloor$ for all real $n\geq 0$ since $\lfloor \sqrt[k]{n}\rfloor = \lfloor \sqrt[k]{\lfloor n\rfloor}\rfloor$ for $n\geq 0$ (see \cite[Equation 3.9]{graham:concrete_math:1994}).}

For $P$ a logical statement on the natural numbers, let us define $f_P(n)$ as the $n$-th positive natural number $m$ such that $P(m)$ is true. For the case where $P(m)$ denotes the logical statement ``$m$ is square'', $f_P(n)$ is easily determined, since the list of integers $m$ such that $P(m)$ is true are easily enumerated. As noted in the example above, for this particular $P$, the function $f_{\neg P}(n)$ can also be computed by an explicit formula.
However, in general, the simplicity of $f_P$ does not imply the simplicity of $f_{\neg P}$. Furthermore, for more general statements $P$, the formula for $f_P(n)$ or $f_{\neg P}(n)$ may not be readily available. Even if such explicit formulas are available, some of them require the use of floating point arithmetic and it can be difficult to use computationally to find $f_P(n)$ or $f_{\neg P}(n)$, especially for large $n$. See for example the formulas for the $n$-th non-Fibonacci number in \cite{gould1965,Farhi2011} which require $\log_\phi$ at high precision for large $n$. While there have been many studies of explicit formulas for such complementary sequences \cite{Lambek1954,gould1965,Honsberger1970,Nelson1988,Reis1990,Nyblom2002,mortici2010,Farhi2011}, there have not been much study
in computer algorithms to calculate such sequences. The purpose of this paper is to discuss algorithms to compute $f_P(n)$ or $f_{\neg P}(n)$.

\section{Finding $f_P(n)$ as the Solution to a Fixed Point Problem}
For an integer $a$, define the {\em counting function} $C_P(a) = |\{b\in\mathbb{N}:(1\leq b\leq a) \wedge P(b)\}|$ as the number of positive integers less than or equal to $a$ such that $P(a)$ is true. 
It is clear that $C_P(a)$ is increasing, $0\leq C_P(a)\leq a$, and $0\leq C_{\neg P}(a)=a-C_{P}(a)\leq a$. Furthermore, $f_P(n)$ is the smallest integer $m$ such that $C_P(m) = n$. Also note that $f_P(n)\geq n$ and $f_P$  is strictly increasing. 

Define $g_n(x) = n+C_{\neg P}(x) = n+x-C_P(x) $. A fixed point $x$ of $g_n$ satisfies $x=n+x-C_P(x)$, i.e., $C_P(x) = n$. Thus, the smallest fixed point of $g_n$ is equal to $f_P(n)$. Furthermore, a fixed point of $g_n$ that is in the range of $f_P$ is equal to $f_P(n)$. In particular, if $g_n$ has a unique fixed point, then it must necessarily be equal to $f_P(n)$. Finding a fixed point of $f(x)$ is equivalent to finding a root of $f(x)-x$. Equivalently, we could define $\tilde{g}_n(x) = n-C_P(x)$ and find the roots of $\tilde{g}_n$. However, in the sequel we will consider the fixed point formulation as  $g_n$ is defined with $C_{\neg P}$  and has a more natural interpretation for complementary sequences. Furthermore, we show below that the function iteration method solving this fixed point problem is equivalent to the well-known Lambek-Moser method for defining complementary sequences.

The function $g_n$, viewed as a function on the real numbers, is a piecewise-linear function. For each value of $n$, the function $g_n$ lines up with the identity function on the segment  $\{m\in\mathbb{N}: C_P(m) = n\}$. It is clear that $g_n(m) > m$ if $m<f_P(n)$ and $g_n(m)\leq m$ for $m\geq f_P(n)$. As an example, we show in Figure \ref{fig:cp} the function $g_n$ for the case where $P(m)$ denotes the statement ``$m$ is prime'' and $n$ is equal to $4$. Notice that the minimal fixed point is at $x=7$ which corresponds to $f_P(4)$, i.e., the fourth prime number.
Next, let us consider methods to find the smallest fixed point of such an increasing piecewise-linear function $g_n$ on the integers.

\begin{figure}[htbp]
\begin{center}
\includegraphics[width=6in]{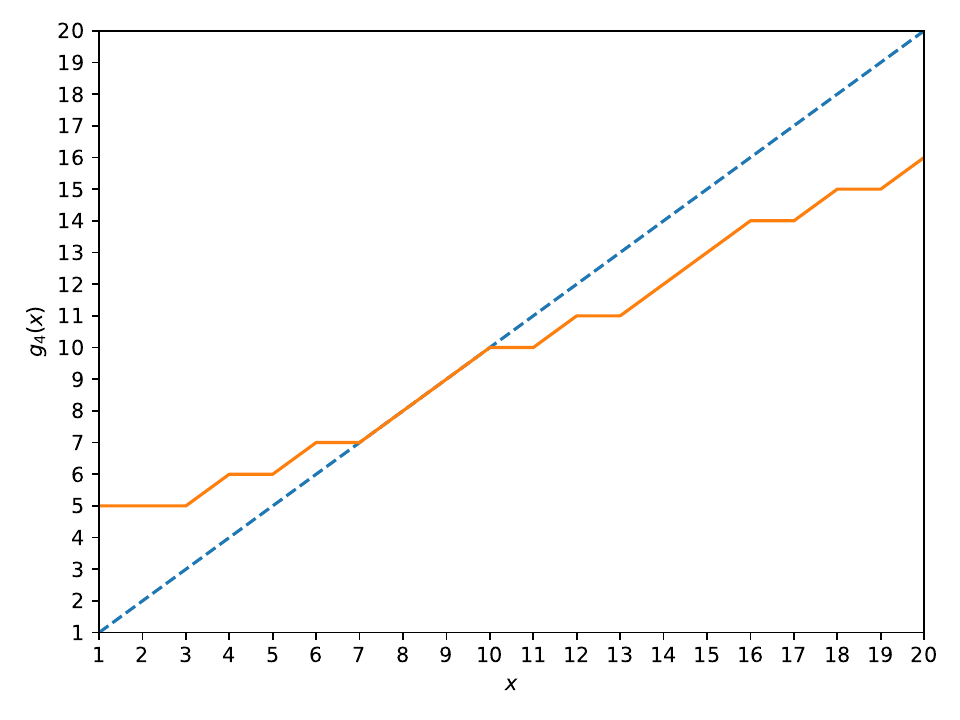}
\end{center}
\caption{$g_n(x)$ when $n=4$ and $P(m)$ denote the logical statement ``$m$ is prime''. The minimal fixed point is at $x=7$ which corresponds to $f_P(4)$, the fourth prime number.}
\label{fig:cp}
\end{figure}

\subsection{Function Iteration Method}
\label{sec:iteration}
The function iteration method to find a fixed point of a function is a classical method that dates back to at least Heron's method for finding an approximation to the square root \cite{Heath1921} and is used in general root finding algorithms.  
We first pick an initial condition less than or equal to $f_P(n)$.
Since $n\leq f_P(n)$, we can start with the initial condition $x = n$ and apply the iteration $x\rightarrow g_n(x)$ repeatedly until convergence (Algorithm \ref{alg:iterationCP}). This is for example implemented in the \texttt{FixedPoint} function in Mathematica.
Note that since $g_n(x) \geq x$ initially, at each step of the algorithm the value of $x$ increases, until it reaches a point where $g_n$ intersects with the identity function, which is the smallest fixed point, i.e., $f_P(n)$.
For the initial condition $x=n$, it is easy to see that this method is equivalent to the Lambek-Moser method and \cite{Lambek1954} showed that it indeed converges to the smallest fixed point.

\begin{algorithm}[htbp]
\begin{algorithmic}
\Require $g_n(x)$
\Comment computes the minimal fixed point of $g_n$.
\State $m \gets n$
\While {$g_n(m) \neq m$}
\State $m \gets g_n(m)$ 
\EndWhile
\State \Return $m$
\end{algorithmic}
\caption{Function iteration method on $g_n(x)$ to compute $f_P(n)$.}
\label{alg:iterationCP}
\end{algorithm}

While the Lambek-Moser method assumes the initial condition $x=n$, depending on $P$ we may choose a more suitable initial condition. 
For instance, if $P(m)$ denotes the statement ``$m$ is the product of $k$ distinct primes'', then the initial condition can be chosen as $\max(n,p_k\#)$ since $f_P(n)\geq p_k\#$ where $p_k\#$  is the $k$-th primorial. 

The number of steps needed for convergence is less than $f_{P}(n)-n$ and thus this algorithm is efficient when $f_{P}(n)-n$ is small with respect to $n$, i.e., when the numbers satisfying $P$ are dense.
In particular, in \cite{Lambek1954} it is shown that if the difference function of $f_{\neg P}$ has at least a linear growth rate (implying that $f_{\neg P}$ grows at least quadratically) then $2$ steps suffice.  More precisely,  it is shown that 
\begin{theorem}
\label{thm:twostep}
If $f_{\neg P}(m+1)-f_{\neg P}(m)\geq m$ for all $m$, then 
$f_P(n) = g_n(g_n(n)) = n+C_{\neg P}(n+C_{\neg P}(n))$.
\end{theorem}
Sequences satisfying these conditions include non-$k$-th-powers or the non-powers of $k$.
Thus, in these cases the computation of $f_P(n)$ requires at most $2$ evaluations of the counting function $C_{\neg P}$. We next show that $1$ evaluation suffices.

\begin{theorem} \label{thm:onestep}
If $f_{\neg P}(m+1)-f_{\neg P}(m)\geq m$ for all $m$, then $f_{P}(n) = n+C_{\neg P}(n)+1$ if $n+C_{\neg P}(n)\geq f_{\neg P}(C_{\neg P}(n)+1)$ and $f_{P}(n) = n+C_{\neg P}(n)$ otherwise.
\end{theorem}
\begin{proof}
First note that $f_{\neg P}(C_{\neg P}(n)+1)> n$. By hypothesis, $f_{\neg P}(C_{\neg P}(n)+2) \geq C_{\neg P}(n)+1+f_{\neg P}(C_{\neg P}(n)+1) > C_{\neg P}(n)+n = g_n(n)$.
This means that $C_{\neg P}(n)+2 > C_{\neg P}(C_{\neg P}(n)+n) = g_n(g_n(n))-n$, i.e., $g_n(g_n(n)) < n+C_{\neg P}(n)+2$.
Since $C_{\neg P}$ is increasing, $g_n(g_n(n))\geq n+C_{\neg P}(n)$.
Finally, it is easy to see that the threshold where $g_n(g_n(n))$ changes from $n+C_{\neg P}(n)$ to $n+C_{\neg P}(n)+1$ is precisely given by $n+C_{\neg P}(n)\geq f_{\neg P}(C_{\neg P}(n)+1)$.
\end{proof}

This can be more compactly expressed using the Iverson bracket \cite{iverson1962,knuth:notation:1992}, which we denote using $\llbracket\hspace{0.1cm}\rrbracket$:
$$f_{P}(n) = n+C_{\neg P}(n)+\llbracket n+C_{\neg P}(n)\geq f_{\neg P}(C_{\neg P}(n)+1)\rrbracket.$$
Next we show several applications of Theorem \ref{thm:onestep}.

\subsubsection{Non-$k$-th-powers}
As an example of applying Theorem \ref{thm:onestep}, we give the following formula for the $n$-th non-$k$-th power for $k>1$, which simplifies the formula given in Section \ref{sec:intro} by requiring only one evaluation of the integer $k$-th root function $\lfloor\sqrt[k]{n}\rfloor$:
\begin{equation*} a(n) = \left\{ \begin{array}{ll} n+\lfloor\sqrt[k]{n}\rfloor+1 & \mbox{if } \quad n+\lfloor\sqrt[k]{n}\rfloor\geq \left(\lfloor\sqrt[k]{n}\rfloor +1\right)^k\\
 n+\lfloor\sqrt[k]{n}\rfloor & \mbox{otherwise.} 
 \end{array} 
 \right.
 \label{eqn:oneeval}
 \end{equation*}

\subsubsection{Non-Mersenne Numbers} 
Similarly, for the non-Mersenne numbers (i.e., numbers not of the form $2^p-1$ for $p$ prime), we have for $n>1$ the formula
  $$ a(n) = \left\{ \begin{array}{ll} n+s+1 & \mbox{if } \quad n+1+s\geq 2^{p_{s +1}}\\
 n+s & \mbox{otherwise,} 
 \end{array} 
 \right.$$
 where $s = \lfloor\log_2{\pi(n)}\rfloor$, $p_k$ is the $k$-th prime, and $\pi(n)$ is the prime counting function that returns the number of prime numbers less than or equal to $n$.

\subsubsection{Non-Fermat Numbers} 
The {\em Fermat numbers} are defined as $2^{2^{n-1}}+1$ for $n\geq 1$, i.e., $3, 5, 17, 257, 65537,\ldots $.
The first two non-Fermat numbers are $1$ and $2$ and the $n$-th non-Fermat numbers for $n>2$
are 
$$ a(n) = \left\{ \begin{array}{ll} n+\lfloor\log_2(\log_2(n-1))\rfloor+2 & \mbox{if } n+\lfloor\log_2(\log_2(n-1))\rfloor\geq 4^{2^{\lfloor\log_2(\log_2(n-1))\rfloor}}\\
  n+\lfloor\log_2(\log_2(n-1))\rfloor+1 & \mbox{otherwise.} 
 \end{array} 
 \right.$$

\subsubsection{Non-powers of $k$}
Similarly, for the non-powers of $k$ we have the formula
$$ a(n) = \left\{ \begin{array}{ll} n+1+\lfloor\log_k{n}\rfloor+1 & \mbox{if } \quad n+1+\lfloor\log_k{n}\rfloor\geq k^{\lfloor\log_k{n}\rfloor +1}\\
 n+1+\lfloor\log_k{n}\rfloor & \mbox{otherwise.} 
 \end{array} 
 \right.$$
 
\subsubsection{Non-Jacobsthal Numbers}
The {\em Jacobsthal numbers} $0,1,1,3,5,11,21,\ldots $ are defined recursively as the sequence $a(n)=n$ for $n\leq 1$ and $a(n) = a(n-1)+2a(n-2)$ otherwise. It can also be defined as the nearest integer to $\frac{2^n}{3}$. This sequence is sequence \href{https://oeis.org/A001045}{A001045} of the on-line encyclopedia of integer sequences (OEIS) \cite{oeis}.  Looking only at the positive integers and ignoring the duplicate $1$, we obtain the sequence $1,3,5,11,21,\ldots$ whose definition can be written as $f_{\neg P}(n) = \left\lfloor\frac{2^{n+1}+1}{3}\right\rfloor$.
The number of such numbers less than or equal to $n$ is $1,1,2,2,3,...$ which can be written as $C_{\neg P}(n) = \left\lfloor\log_2(3n+1)\right\rfloor-1$. Applying this to Theorem \ref{thm:onestep} results in the following formula for the non-Jacobsthal numbers (OEIS \href{https://oeis.org/A147613}{A147613}):

$$a(n) = \left\{ \begin{array}{ll} n+ \left\lfloor\log_2(3n+1)\right\rfloor &\mbox{if } \quad  n+ \left\lfloor\log_2(3n+1)\right\rfloor > \left\lfloor\frac{2^{\left\lfloor\log_2(3n+1)\right\rfloor+1}+1}{3}\right\rfloor \\
n+ \left\lfloor\log_2(3n+1)\right\rfloor-1 &\mbox{otherwise.}
\end{array}
\right.$$

\subsection{Interleaving Functions} \label{sec:interleaving}
Typically, the computation of $C_{\neg P}$ requires the inversion of $f_{\neg P}$, which can be difficult to do. For instance, if $f_{\neg P}(n)$ is a polynomial in $n$ of degree $5$ or more then by the Abel-Ruffini theorem it is in general not solvable in radicals.
However, if the sequence $\{f_{\neg P}(n)\}$ is interleaved with another sequence $\{\alpha(n)\}$ that is more easily invertible (either analytically or computationally), then we can leverage this to more efficiently find the complementary sequence $f_{P}$. More precisely, assume that we are
given a real-valued
increasing function $\alpha$ such that $\alpha(1)=1$ and
$f_{\neg P}(m-1)\leq \alpha(m)\leq f_{\neg P}(m)$ for all $m$. Note that we do not require that $\alpha$ is integer-valued. We define the increasing and integer-valued function $h(n) = \max\{m\in \mathbb{N}^+:\alpha(m)\leq n\}$. The idea is to choose $\alpha$ such that $h$ is easier to compute than inverting $f_{\neg P}$ .
We can then compute $f_{P}(n)$ with one evaluation of $h(n)$.

\begin{theorem}
If $f_{\neg P}$ is such that $f_{\neg P}(m+1)-f_{\neg P}(m)\geq m$ for all $m$, and $\alpha$, $h$ are as defined above, then
\begin{equation}  
f_{P}(n) = \left\{ \begin{array}{ll} n+h(n)+1 & \text{if } \quad n+h(n)\geq f_{\neg P}(h(n)+1)\\
n+h(n)-1 & \text{if } \quad n+h(n)\leq f_{\neg P}(h(n))\\
n+h(n) & \text{otherwise.} 
\end{array} 
\right.
\label{eqn:interleave}
\end{equation}
\label{thm:interleave}
\end{theorem}
 
\begin{proof}
The conditions on $f_{\neg P}$ imply that we can apply Theorem \ref{thm:onestep}. Next note that 
$ n+h(n)\geq f_{\neg P}(h(n)+1)\geq f_{\neg P}(h(n))+h(n)$ implies $f_{\neg P}(h(n))\leq n$ and that
$n+h(n)\leq f_{\neg P}(h(n))$ implies $n< f_{\neg P}(h(n))$ since $h(n)\geq 1$.
The result then follows from the fact that $C_{\neg P}(n) = h(n)$ if $f_{\neg P}(h(n))\leq n$ and $C_{\neg P}(n) = h(n)-1$ otherwise.
\end{proof} 

This result can be further simplified depending on how close $f_{\neg P}$ and $\alpha$ are.

\begin{corollary}
Given the hypothesis of Theorem \ref{thm:interleave},
if $f_{\neg P}(m)-\alpha(m)< m$, then
\begin{equation*} f_{P}(n) = \left\{ \begin{array}{ll} n+h(n)+1 & \mbox{if } \quad n+h(n)\geq f_{\neg P}(h(n)+1)\\
 n+h(n) & \mbox{otherwise.}
 \end{array} 
 \right.
 \label{eqn:interleave2a}
 \end{equation*}
 \label{cor:interleave2a}
\end{corollary}
\begin{proof}
Note that by definiton of $h$, $\alpha(h(n))\leq n$.
Out of the three conditions in Equation (\ref{eqn:interleave}), the second condition is never satisfied by hypothesis since otherwise it reaches the contradiction  $n+h(n)-1 < f_{\neg P}(h(n)) \leq h(n)+\alpha(h(n))-1 \leq n+h(n)-1$.
\end{proof}

\begin{corollary}
Given the hypothesis of Theorem \ref{thm:interleave},
if $f_{\neg P}(m)-\alpha(m)\geq m$ for $m>1$, then
\begin{equation*} f_{P}(n) = \left\{ \begin{array}{ll} n+h(n) & \mbox{if } \quad n+h(n)> f_{\neg P}(h(n))\\
 n+h(n)-1 & \mbox{otherwise.}
 \end{array} 
 \right.
 \label{eqn:interleave2b}
 \end{equation*}
 \label{cor:interleave2b}
\end{corollary}
\begin{proof}
Note that by definiton of $h$, $\alpha(h(n)+1)> n$.
Out of the three conditions in Equation (\ref{eqn:interleave}), the first condition is never satisfied by hypothesis since otherwise it reaches the contradiction  $n+h(n)\geq f_{\neg P}(h(n)+1)\geq h(n)+\alpha(h(n)+1)>n+h(n)$.
\end{proof}

To illustrate, we will use these results to obtain formulas for the non-$k$-gonal numbers, the non-$k$-gonal-pyramidal numbers, the non-$k$-simplex numbers, the non-sum-of-$k$-th-powers and the non-centered-$k$-gonal numbers. Furthermore, we choose $h$ and $\alpha$ such that these formulas can be implemented algorithmically using integer arithmetic.

\subsubsection{Non-$k$-gonal Numbers} \label{sec:polygonal}
The {\em $n$-th $k$-gonal numbers} are defined as $T(k,n) = \frac{(k-2)n(n-1)}{2}+n$ for $k\geq 2$ and $n\geq 1$. For $k=2$, the $2$-gonal numbers are simply the natural numbers and thus there are no non-$2$-gonal numbers. 
In this section we give a general formula for the $n$-th non-$k$-gonal number for $k\geq 3$. 

\begin{theorem}
\label{thm:polygonal}
The $n$-th non-$k$-gonal number ($k\geq 3$) is given by
\begin{equation} a(n) = \left\{ \begin{array}{ll} n+\left\lfloor\sqrt{\frac{2n}{k-2}}\right\rfloor+1 & \mbox{if } \quad 2n > (k-2)\left\lfloor\sqrt{\frac{2n}{k-2}}\right\rfloor\left(\left\lfloor\sqrt{\frac{2n}{k-2}}\right\rfloor +1\right)\\[5pt]
 n+\left\lfloor\sqrt{\frac{2n}{k-2}}\right\rfloor & \mbox{otherwise,} 
 \end{array} 
 \right.
 \label{eqn:polygonal}
 \end{equation}
  i.e., $$a(n) = n+\left\lfloor\sqrt{\frac{2n}{k-2}}\right\rfloor+\left\llbracket 2n > (k-2)\left\lfloor\sqrt{\frac{2n}{k-2}}\right\rfloor\left(\left\lfloor\sqrt{\frac{2n}{k-2}}\right\rfloor +1\right)\right\rrbracket.$$
For $3\leq k\leq 10$, this can be written as $a(n) = n+\left\lfloor\sqrt{\frac{2n}{k-2}}+\frac{1}{2}\right\rfloor$.
 \end{theorem}
\begin{proof}
Note that $T(k,n+1)-T(k,n) = (k-2)n+1\geq n$ for $k\geq 3$.
Furthermore, $T(k,n)$ is interleaved with $\alpha(n) = \frac{k-2}{2}(n-1)^2$ with corresponding $h(n) = \left\lfloor \sqrt{\frac{2n}{k-2}} \right\rfloor + 1$.
Since $T(k,n)-\alpha(n)=\frac{k}{2}(n-1)+1\geq n$, we can apply Corollary \ref{cor:interleave2b} and obtain Equation (\ref{eqn:polygonal}).

Assume that $k\leq 10$. The inequality
$\sqrt{\frac{2n}{k-2}}- \left\lfloor \sqrt{\frac{2n}{k-2}} \right\rfloor\geq \frac{1}{2}$ is equivalent to $2n\geq (k-2) \left\lfloor \sqrt{\frac{2n}{k-2}} \right\rfloor\left( \left\lfloor \sqrt{\frac{2n}{k-2}} \right\rfloor+1\right)+\frac{k-2}{4}$. Since $2n$ and $(k-2) \left\lfloor \sqrt{\frac{2n}{k-2}} \right\rfloor \left( \left\lfloor \sqrt{\frac{2n}{k-2}} \right\rfloor+1\right)$ are both even and $\frac{k-2}{4}\leq 2$ this is equivalent to $2n\geq (k-2) \left\lfloor \sqrt{\frac{2n}{k-2}} \right\rfloor\left( \left\lfloor \sqrt{\frac{2n}{k-2}} \right\rfloor+1\right)+2$ which in turn is equivalent to  $2n>(k-2) \left\lfloor \sqrt{\frac{2n}{k-2}} \right\rfloor\left( \left\lfloor \sqrt{\frac{2n}{k-2}} \right\rfloor+1\right)$ which is exactly the condition in 
Equation (\ref{eqn:polygonal}).
\end{proof}

Recall that $\left\lfloor \sqrt{\frac{2n}{k-2}} \right\rfloor$ can be computed as $\left\lfloor \sqrt{\left\lfloor\frac{2n}{k-2}\right\rfloor} \right\rfloor$. For instance in Python 3.x this can be implemented as \verb+isqrt(2*n//(k-2))+.
By setting $k=3$ or $k=4$ to Theorem \ref{thm:polygonal}, we get the following Corollary:
\begin{corollary} \label{cor:squaretriangular}
The $n$-th non-triangular number is given by $n+\lfloor\sqrt{2n}+\frac{1}{2}\rfloor$.
The $n$-th non-square number is given by $n+\lfloor\sqrt{n}+\frac{1}{2}\rfloor$.
\end{corollary}

In discussing \cite[Example 4]{Lambek1954}, it was shown that the $n$-th non-square number is 
$n+\lfloor\sqrt{n}+\frac{1}{2}\rfloor$, while discussing \cite[Example 6]{Lambek1954} it was shown that $n+\lfloor\sqrt{n+\lfloor\sqrt{n}\rfloor}\rfloor$ is a new formula for the $n$-th non-square number. Theorem \ref{thm:twostep} along with Corollary \ref{cor:squaretriangular} show the equivalence of these two formulas. 

In \cite[Example 5]{Lambek1954} it was reported that the $n$-th non-triangular number is $n+\lfloor \sqrt{2n}+\frac{1}{2}\rfloor$ which also follows from Corollary  \ref{cor:squaretriangular}.

\begin{theorem}
\label{thm:polygonal2}
For $k\geq 3$, let $r$ be a real number such that $-1\leq r\leq k-4$. Then the $n$-th non-$k$-gonal number is given by
\begin{equation} a(n) = \left\{ \begin{array}{ll} n+\left\lfloor\sqrt{\frac{2n+r}{k-2}}\right\rfloor+1 & \mbox{if } \quad 2n > (k-2)\left\lfloor\sqrt{\frac{2n+r}{k-2}}\right\rfloor\left(\left\lfloor\sqrt{\frac{2n+r}{k-2}}\right\rfloor +1\right)\\[5pt]
 n+\left\lfloor\sqrt{\frac{2n+r}{k-2}}\right\rfloor & \mbox{otherwise.} 
 \end{array} 
 \right.
 \label{eqn:polygonal2}
 \end{equation}
By choosing $r=\left\lfloor\frac{k+1}{4}\right\rfloor -2$, this can be simplified as $a(n) = n+\left\lfloor\sqrt{\frac{2n-2+\left\lfloor\frac{k+1}{4}\right\rfloor}{k-2}}+\frac{1}{2}\right\rfloor$.
 \end{theorem}
\begin{proof}
Let $\alpha(n) = \frac{k-2}{2}(n-1)^2-\frac{r}{2}$ with corresponding $h(n) = \left\lfloor \sqrt{\frac{2n+r}{k-2}} \right\rfloor + 1$.
Note that since $r\geq -1$, $T(k,n)-\alpha(n)=\frac{k}{2}(n-1)+1+\frac{r}{2}\geq \frac{k}{2}(n-1)+\frac{1}{2} \geq n$ for $n>1$. Furthermore, $\alpha(n+1)-T(k,n) = \frac{(k-4)n}{2}-\frac{r}{2}\geq 0$
so we can apply Corollary \ref{cor:interleave2b} and obtain Equation (\ref{eqn:polygonal2}).

Next,
$\sqrt{\frac{2n+r}{k-2}}- \left\lfloor \sqrt{\frac{2n+r}{k-2}} \right\rfloor\geq \frac{1}{2}$ is equivalent to 
$$2n\geq (k-2) \left\lfloor \sqrt{\frac{2n+r}{k-2}} \right\rfloor\left( \left\lfloor \sqrt{\frac{2n+r}{k-2}} \right\rfloor+1\right)+\frac{k-2}{4}-r.$$ By picking $r = \left\lceil\frac{k-10}{4}\right\rceil =  \left\lfloor\frac{k+1}{4}\right\rfloor -2 $, this ensures
that $\frac{k-2}{4}-r\leq 2$. Furthermore, $r\geq -1$ for $k\geq 3$. Since $2n$ and $(k-2) \left\lfloor \sqrt{\frac{2n+r}{k-2}} \right\rfloor \left( \left\lfloor \sqrt{\frac{2n+r}{k-2}} \right\rfloor+1\right)$ are both even this is equivalent to $2n\geq (k-2) \left\lfloor \sqrt{\frac{2n+r}{k-2}} \right\rfloor\left( \left\lfloor \sqrt{\frac{2n+r}{k-2}} \right\rfloor+1\right)+2$ which in turn is equivalent to $2n>(k-2) \left\lfloor \sqrt{\frac{2n+r}{k-2}} \right\rfloor\left( \left\lfloor \sqrt{\frac{2n+r}{k-2}} \right\rfloor+1\right)$, i.e., the condition in 
Equation (\ref{eqn:polygonal2}).
\end{proof}

A consequence of Theorems \ref{thm:polygonal} and \ref{thm:polygonal2} is that for 
$3\leq k\leq 6$,  the $n$-th non-$k$-gonal number is equal to $n+\left\lfloor\sqrt{\frac{2n}{k-2}}+\frac{1}{2}\right\rfloor=n+\left\lfloor\sqrt{\frac{2n-1}{k-2}}+\frac{1}{2}\right\rfloor$.

\subsubsection{Non-second-$k$-gonal Numbers}\label{sec:second-polygonal}
For $k\geq 5$, the {\em $n$-th second $k$-gonal number} is defined as $T_2(k,n) = \frac{k-2}{2}n^2+\frac{k-4}{2}n$. The first differences are given by $T_2(k,n+1)-T_2(k,n)=(k-2)n+k-3\geq n$.
We choose $\alpha(n) = \frac{k-2}{2}n^2$ with $h(n) = \left\lfloor\sqrt{\frac{2n}{k-2}}\right\rfloor$.
Since $T_2(k,n)-\alpha(n)=\frac{k-4}{2}n$, we can apply Corollary \ref{cor:interleave2a} for $k=5$ and apply Corollary \ref{cor:interleave2b} for $k>5$ and obtain the following result.

\begin{theorem}
\label{thm:secondpentagonal}
The $n$-th non-second-pentagonal number is given by
\begin{equation*} a(n) = \left\{ \begin{array}{ll} n+\left\lfloor\sqrt{\frac{2n}{3}}\right\rfloor+1 & \mbox{if } \quad 2n \geq \left\lfloor\sqrt{\frac{2n}{3}}\right\rfloor\left(3\left\lfloor\sqrt{\frac{2n}{3}}\right\rfloor +5\right)+4\\[5pt]
 n+\left\lfloor\sqrt{\frac{2n}{3}}\right\rfloor & \mbox{otherwise.} 
 \end{array} 
 \right.
 \label{eqn:secondpentagonal}
 \end{equation*}
 \end{theorem}

\begin{theorem}
\label{thm:secondpolygonal}
The $n$-th non-second-$k$-gonal number ($k>5$) is given by
\begin{equation} a(n) = \left\{ \begin{array}{ll} n+\left\lfloor\sqrt{\frac{2n}{k-2}}\right\rfloor & \mbox{if } 2n > (k-2)\left\lfloor\sqrt{\frac{2n}{k-2}}\right\rfloor\left(\left\lfloor\sqrt{\frac{2n}{k-2}}\right\rfloor +1\right)-4\left\lfloor\sqrt{\frac{2n}{k-2}}\right\rfloor\\[5pt]
 n+\left\lfloor\sqrt{\frac{2n}{k-2}}\right\rfloor-1 & \mbox{otherwise.} 
 \end{array} 
 \right.
 \label{eqn:secondpolygonal}
 \end{equation}
 \end{theorem}

For $k=6$, the condition in Equation (\ref{eqn:secondpolygonal}) becomes $n>2\lfloor\sqrt{\frac{n}{2}}\rfloor^2$ which is satisfied if and only if $\frac{n}{2}$ is not a square. Thus, the $n$-th non-second-hexagonal number is $n+\lfloor\sqrt{\frac{n}{2}}\rfloor-1$ if $\frac{n}{2}$ is a square and equal to $n+\lfloor\sqrt{\frac{n}{2}}\rfloor$ otherwise. Equivalently, the $n$-th non-second-hexagonal number is  $n+\lceil\sqrt{\frac{n}{2}}\rceil - 1$.

\subsubsection{Non-centered-$k$-gonal Numbers}
The {\em $n$-th centered $k$-gonal number} is defined as $\frac{kn(n+1)}{2}+1$.
Using $\alpha(m) = \frac{km^2}{2}$ and the corresponding $h(n) = \left\lfloor\sqrt{\frac{2n}{k}}\right\rfloor$, and the fact that $\frac{kn(n+1)}{2}+1-\frac{kn^2}{2} = \frac{kn}{2}+1\geq n$,
Corollary \ref{cor:interleave2b} can be used to show the following result.
\begin{theorem}\label{thm:centered}
The $n$-th non-centered-$k$-gonal number is
\begin{equation*} a(n) = \left\{ \begin{array}{ll} n+\left\lfloor\sqrt{\frac{2n}{k}}\right\rfloor & \mbox{if } \quad 2\left(n+\left\lfloor\sqrt{\frac{2n}{k}}\right\rfloor-1\right)> k\left\lfloor\sqrt{\frac{2n}{k}}\right\rfloor \left(\left\lfloor\sqrt{\frac{2n}{k}}\right\rfloor+1\right)\\[5pt]
 n+\left\lfloor\sqrt{\frac{2n}{k}}\right\rfloor-1 & \mbox{otherwise.} 
 \end{array} 
 \right.
 \label{eqn:centered}
 \end{equation*}
i.e.,
$$a(n) = n+\left\lfloor\sqrt{\frac{2n}{k}}\right\rfloor-\left\llbracket 2\left(n+\left\lfloor\sqrt{\frac{2n}{k}}\right\rfloor-1\right)\leq k\left\lfloor\sqrt{\frac{2n}{k}}\right\rfloor \left(\left\lfloor\sqrt{\frac{2n}{k}}\right\rfloor+1\right)\right\rrbracket.$$
\end{theorem}
Since the $k$-gonal numbers are defined starting from $n=0$, the formula above will compute the $n$-th nonnegative
non-centered-$k$-gonal number, i.e., it returns $0$ when $n=1$ and lists the positive
non-centered-$k$-gonal numbers starting from $n=2$.
For instance, for $k=4$, the $n$-th non-centered-square number (which except for the first term is OEIS sequence \href{https://oeis.org/A350757}{A350757})
is equal to 
\begin{equation*}  \left\{ \begin{array}{ll} n+\left\lfloor\sqrt{\frac{n}{2}}\right\rfloor & \mbox{if } \quad n>\left\lfloor\sqrt{\frac{n}{2}}\right\rfloor\left(2\left\lfloor\sqrt{\frac{n}{2}}\right\rfloor+1\right)+1\\[5pt]
n+\left\lfloor\sqrt{\frac{n}{2}}\right\rfloor- 1& \mbox{otherwise.}
\end{array}\right.
\end{equation*}

\subsubsection{Non-$k$-gonal-pyramidal Numbers}\label{sec:pyramidal}
For $k\geq 3$, the {\em $n$-th $k$-gonal pyramidal number} is defined as 
$$f_{\neg P}(n) = \frac{n(n+1)(n(k-2)-(k-5))}{6}.$$
By choosing $\alpha(m) = \frac{(k-2)m^3}{6}$, we have 
$h(n) =  \left\lfloor\sqrt[3]{\frac{6n}{k-2}}\right\rfloor$ and $$\alpha(n+1)-f_{\neg P}(n) = \frac{(n + 1)(k + n(3k - 9) - 2)}{6} \geq 0.$$
Since $f_{\neg P}(n)-\alpha(n) = \frac{n(3n+5-k)}{6}\geq 0$ for $k\leq 8$, this implies that we can apply Theorem \ref{thm:interleave} to obtain the following result.

\begin{theorem}\label{thm:pyramidal}
For $3\leq k\leq 8$, the $n$-th non-$k$-gonal-pyramidal number is
\begin{equation} a(n) = \left\{ \begin{array}{lll} 
n+\left\lfloor\sqrt[3]{\frac{6n}{k-2}}\right\rfloor+1 & \mbox{if } \quad 6n\geq &\left\lfloor\sqrt[3]{\frac{6n}{k-2}}\right\rfloor^3(k - 2) 
\\&&+ 3\left\lfloor\sqrt[3]{\frac{6n}{k-2}}\right\rfloor^2(k - 1) 
\\[5pt]&&+ \left\lfloor\sqrt[3]{\frac{6n}{k-2}}\right\rfloor(2k - 1) + 6
\\[5pt]
n+\left\lfloor\sqrt[3]{\frac{6n}{k-2}}\right\rfloor-1 & \mbox{if } \quad 6n\leq & \left\lfloor\sqrt[3]{\frac{6n}{k-2}}\right\rfloor\left(\left\lfloor\sqrt[3]{\frac{6n}{k-2}}\right\rfloor - 1\right)\times\\[5pt]&&\left(\left\lfloor\sqrt[3]{\frac{6n}{k-2}}\right\rfloor(k - 2) + k+1\right) \\
n+\left\lfloor\sqrt[3]{\frac{6n}{k-2}}\right\rfloor & \mbox{otherwise.} &
 \end{array} 
 \right.
 \label{eqn:pyramidal}
 \end{equation}
\end{theorem}

If $k\leq 5$, then $f_{\neg P}(n)-\alpha(n)\geq n$ and we can apply Corollary \ref{cor:interleave2b} to obtain the following result.
\begin{corollary}\label{cor:pyramidal}
For $3\leq k\leq 5$, the $n$-th non-$k$-gonal-pyramidal number is
\begin{equation} a(n) = \left\{ \begin{array}{lll} 
n+\left\lfloor\sqrt[3]{\frac{6n}{k-2}}\right\rfloor-1 & \mbox{if } \quad 6n\leq &\left\lfloor\sqrt[3]{\frac{6n}{k-2}}\right\rfloor\left(\left\lfloor\sqrt[3]{\frac{6n}{k-2}}\right\rfloor - 1\right)\times
\\[5pt]&&\left(\left\lfloor\sqrt[3]{\frac{6n}{k-2}}\right\rfloor(k - 2) + k+1\right) \\
n+\left\lfloor\sqrt[3]{\frac{6n}{k-2}}\right\rfloor & \mbox{otherwise.} &
 \end{array} 
 \right.
 \label{eqn:pyramidal2}
 \end{equation}
\end{corollary}

For instance, the $n$-th non-pentagonal-pyramidal number is given by
\begin{equation*}  \left\{ \begin{array}{ll} 
n+\lfloor\sqrt[3]{2n}\rfloor & \text{if } \quad 2n>\lfloor\sqrt[3]{2n}\rfloor(\lfloor\sqrt[3]{2n}\rfloor-1)(\lfloor\sqrt[3]{2n}\rfloor +2)\\
n+\lfloor\sqrt[3]{2n}\rfloor-1 &  \mbox{otherwise.} 
 \end{array} 
 \right.
 \end{equation*}
For other values of $k$, Equations (\ref{eqn:pyramidal}-\ref{eqn:pyramidal2}) in Theorem \ref{thm:pyramidal} and Corollary \ref{cor:pyramidal} still hold for all $n\geq n_0(k)$ for some $n_0(k)>0$. On the other hand, checking for all $n< n_0(k)$ shows that Equation (\ref{eqn:pyramidal2}) (and thus Corollary \ref{cor:pyramidal}) holds for $3\leq k\leq 8$.

\begin{conjecture} \label{conj:pyramidal}
For $k\geq 9$, the $n$-th non-$k$-gonal-pyramidal number is given by Equation (\ref{eqn:pyramidal}).
\end{conjecture}

\subsubsection{Non-sum-of-$k$-th-powers}\label{sec:sum_of_powers}
The sum of $k$-th powers $\sum_{i=1}^n i^k$ can be written as a polynomial of degree $k+1$ by Faulhaber's formula (also known as Bernoulli's formula)
\begin{equation*}
\label{eqn:faulhaber} S(k,n) = \sum_{i=1}^n i^k = \frac{1}{k+1}\sum_{j=0}^k\binom{k+1}{j}B^+_j n^{k+1-j}
\end{equation*}
where $B^+_j$ are the Bernoulli numbers of the second kind with $B^+_1=\frac{1}{2}$. The case $k=1$ corresponds to the triangular numbers which was discussed in Section \ref{sec:polygonal}, so we assume that $k>1$ in this section.
Pascal's identity \cite{MacMillan2011} $(n+1)^{k+1}-1 = \sum_{j=0}^k\binom{k+1}{j}\sum_{i=1}^n i^j$
implies that $S(k,n)+n<\frac{(n+1)^{k+1}}{k+1}$.
Since $x^k$ is a convex function, $\int_{0}^n x^kdx = \frac{n^{k+1}}{k+1}$ is upper bounded by the trapezoidal sum, i.e.,
$\sum_{i=1}^n i^k - \sum_{i=1}^n \frac{i^k-(i-1)^k}{2} = \sum_{i=1}^n i^k-\frac{n^k}{2} \geq \frac{n^{k+1}}{k+1}$
and thus $S(k,n) \geq \frac{n^{k+1}}{k+1}+n$ for $n,k>1$.
Thus, we can pick $\alpha(n) = \frac{n^{k+1}}{k+1}$ and $h(n) = \lfloor\sqrt[k+1]{(k+1)n}\rfloor$ and use Corollary \ref{cor:interleave2b} to show the following result.
\begin{theorem}\label{thm:faulhaber_sum}
The $n$-th non-sum-of-$k$-th-powers (for $k>1$) is
\begin{equation*} a(n) = \left\{ \begin{array}{ll} 
n+\lfloor\sqrt[k+1]{(k+1)n}\rfloor & \text{if } \quad (k+1)(n+\lfloor\sqrt[k+1]{(k+1)n}\rfloor)\\&\quad\quad > \sum_{j=0}^k\binom{k+1}{j}B^+_j \lfloor\sqrt[k+1]{(k+1)n}\rfloor^{k+1-j}\\
 n+\lfloor\sqrt[k+1]{(k+1)n}\rfloor-1 &  \mbox{otherwise,} 
 \end{array} 
 \right.
 \label{eqn:faulhaber_sum}
 \end{equation*}
i.e.,
$$\begin{array}{ll}a(n) = & n+\lfloor\sqrt[k+1]{(k+1)n}\rfloor\\
&-\left\llbracket (k+1)(n+\lfloor\sqrt[k+1]{(k+1)n}\rfloor)\leq \sum_{j=0}^k\binom{k+1}{j}B^+_j \lfloor\sqrt[k+1]{(k+1)n}\rfloor^{k+1-j}\right\rrbracket.
\end{array}$$
or
$$a(n) = n+\lfloor\sqrt[k+1]{(k+1)n}\rfloor-\left\llbracket n\leq \sum_{i=1}^{\lfloor\sqrt[k+1]{(k+1)n}\rfloor} (i^k-1)\right\rrbracket.$$
\end{theorem}

For the case of $k=2$, i.e., the non-square-pyramidal numbers (OEIS \href{https://oeis.org/A302058}{A302058}), 
this results in 
\begin{equation*}  a(n) = \left\{ \begin{array}{ll} 
n+ \lfloor\sqrt[3]{3n}\rfloor& \text{if } \quad 6n> \lfloor\sqrt[3]{3n}\rfloor \left( \lfloor\sqrt[3]{3n}\rfloor-1\right)\left(2 \lfloor\sqrt[3]{3n}\rfloor+5\right)\\
 n+ \lfloor\sqrt[3]{3n}\rfloor-1 &  \mbox{otherwise.} 
 \end{array} 
 \right.
 \end{equation*}

Similarly, the same argument shows the following result.
\begin{theorem}\label{thm:faulhaber_sum_2}
The $n$-th number that is not of the form $m+\sum_{i=1}^m i^k$ (for $k>1$) is given by
\begin{equation*} a(n) = \left\{ \begin{array}{ll} n+\lfloor\sqrt[k+1]{(k+1)n}\rfloor & \text{if } \quad (k+1)n> \\ &\quad\quad \sum_{j=0}^k\binom{k+1}{j}B^+_j \lfloor\sqrt[k+1]{(k+1)n}\rfloor^{k+1-j}\\
 n+\lfloor\sqrt[k+1]{(k+1)n}\rfloor-1 & \mbox{otherwise,}  
 \end{array} 
 \right.
 \label{eqn:faulhaber_sum_2}
 \end{equation*}
i.e.,
$$a(n) = n+\lfloor\sqrt[k+1]{(k+1)n}\rfloor-\left\llbracket (k+1)n\leq \sum_{j=0}^k\binom{k+1}{j}B^+_j \lfloor\sqrt[k+1]{(k+1)n}\rfloor^{k+1-j}\right\rrbracket$$
or
$$a(n) = n+\lfloor\sqrt[k+1]{(k+1)n}\rfloor-\left\llbracket n\leq \sum_{i=1}^{\lfloor\sqrt[k+1]{(k+1)n}\rfloor} i^k\right\rrbracket.$$
\end{theorem}

For instance, the $n$-th number that is not of the form $m+\sum_{i=1}^m i^2$ is given by 
\begin{equation*}  a(n) = \left\{ \begin{array}{ll} 
n+ \lfloor\sqrt[3]{3n}\rfloor& \text{if } \quad 6n> \lfloor\sqrt[3]{3n}\rfloor \left( \lfloor\sqrt[3]{3n}\rfloor+1\right)\left(2 \lfloor\sqrt[3]{3n}\rfloor+1\right)\\
n+ \lfloor\sqrt[3]{3n}\rfloor-1 &  \mbox{otherwise.} 
 \end{array} 
 \right.
 \end{equation*}

\subsubsection{Non-$k$-simplex Numbers}\label{sec:simplex}

\begin{lemma}
\label{lem:factorial}
Let $n^{\overline{k}} = \prod_{i=0}^{k-1}(n+i)$ be the rising factorial. Then 
\begin{equation}\label{eqn:rf}
\left(n+\left\lfloor\frac{k}{2}\right\rfloor-1\right)^k \leq n^{\overline{k}} \leq \left(n+\left\lfloor\frac{k}{2}\right\rfloor\right)^k
\end{equation}
for $n\geq r$, where $r$ is defined as 
$r = \left\lfloor\frac{k^2-4k+6}{4}\right\rfloor$ if $k$ is even and $r = \left\lfloor\frac{k^2-6k+11}{4}\right\rfloor$ if $k$ is odd.
\end{lemma}
\begin{proof}
Let $a = \frac{k-1}{2}$. If $k$ is even, 
$$\begin{array}{lcl}
n^{\overline{k}} &=& \prod_{i=0}^{\frac{k}{2}-1}(n+a-i-\frac{1}{2})(n+a+i+\frac{1}{2})\\
&\leq&  \prod_{i=0}^{\frac{k}{2}-1}((n+a)^2-(i+\frac{1}{2})^2)\\
&\leq&  \prod_{i=0}^{\frac{k}{2}-1}(n+a)^2\\
&\leq& (n+\lfloor\frac{k}{2}\rfloor)^{2\lfloor a\rfloor +2}\\
&\leq& \left(n+\lfloor\frac{k}{2}\rfloor\right)^k
\end{array}$$
Furthermore, since $(n+a)^2-(i+\frac{1}{2})^2 - (n+a-\frac{1}{2})^2 = n+i-\frac{i^2}{4}-1\geq n+k-\frac{k^2}{4}-1\geq n-r\geq 0$ for $i\leq\frac{k}{2}-1$, this implies that  $\left(n+\lfloor\frac{k}{2}\rfloor-1\right)^k \leq n^{\overline{k}}$.

If $k$ is odd,  $n^{\overline{k}} = (n+a)\prod_{i=1}^{a}(n+a-i)(n+a+i)$ and a similar argument as above shows that $n^{\overline{k}} \leq \left(n+\left\lfloor\frac{k}{2}\right\rfloor\right)^k$. Furthermore,
$(n+a)^2-i^2 - (n+a-1)^2 = 2n+\frac{3i}{2}-\frac{i^2}{4}-\frac{9}{4}\geq 2n+\frac{3k}{2}-\frac{k^2}{4}-\frac{9}{4}\geq 2(n-r)\geq 0$ for $i\leq a$ implies that  $\left(n+\lfloor\frac{k}{2}\rfloor-1\right)^k \leq n^{\overline{k}}$.
\end{proof}

Since $r\leq 1$ for $k\leq 5$, this implies that Equation (\ref{eqn:rf}) is true for all $n\geq 1$ and $k\leq 5$. By checking whether Equation (\ref{eqn:rf}) is satisfied for all $n<r$,
we find that Equation (\ref{eqn:rf}) is also true for $n\geq 1$ and $k = 7$ or $k=9$.

The {\em $k$-simplex} (or {\em $k$-polytopic}) numbers are defined as $S(k,n) = \begin{binom}{n+(k-1)}{k}\end{binom}= \frac{n^{\overline{k}}}{k!}$. We are interested in the cases $k>1$. 
By Lemma \ref{lem:factorial} we can choose $\alpha(m) = \frac{(m+\lfloor\frac{k}{2}\rfloor-1)^k}{k!}$ and $h(n) = \lfloor \sqrt[k]{k!n}\rfloor-\lfloor\frac{k}{2}\rfloor+1$.
Note that $S(k,n+1)-S(k,n) = \begin{binom}{n+(k-1)}{k-1}\end{binom} \geq n$.
Theorem \ref{thm:interleave} implies the following Theorem.
\begin{theorem}\label{thm:simplex1}
For $1<k\leq 5$ or $k=7$ or $k=9$, the $n$-th non-$k$-simplex number is

\begin{equation}
a(n) = 
\left\{ \begin{array}{ll} 
n+\lfloor \sqrt[k]{k!n}\rfloor - \lfloor\frac{k}{2}\rfloor+2 & \mbox{if } \quad n+\lfloor \sqrt[k]{k!n}\rfloor- \lfloor\frac{k}{2}\rfloor+1 \geq\\
&\quad\quad\begin{dbinom}{\lfloor \sqrt[k]{k!n}\rfloor+\lceil \frac{k}{2}\rceil+1}{k}\end{dbinom}\\
n+\lfloor \sqrt[k]{k!n}\rfloor - \lfloor\frac{k}{2}\rfloor & \mbox{if }\quad n+\lfloor \sqrt[k]{k!n}\rfloor- \lfloor\frac{k}{2}\rfloor < \begin{dbinom}{\lfloor \sqrt[k]{k!n}\rfloor+\lceil \frac{k}{2}\rceil}{k}\end{dbinom}\\
n+\lfloor \sqrt[k]{k!n}\rfloor - \lfloor\frac{k}{2}\rfloor+1 & \mbox{otherwise.} 
 \end{array} 
 \right.
 \label{eqn:simplex}
 \end{equation}
\end{theorem}

Consider the case $k=3$, i.e., the non-tetrahedral numbers (OEIS \href{https://oeis.org/A145397}{A145397}).
$S(3,n)-\alpha(n) = \frac{n(3n+2)}{6} \geq n$ for $n>1$. Thus, we can apply Corollary \ref{cor:interleave2b} and Equation (\ref{eqn:simplex})
 reduces to 
 
\begin{equation*}
a(n) = 
\left\{ \begin{array}{ll}
n+\lfloor \sqrt[3]{6n}\rfloor & \mbox{if } \quad n+\lfloor \sqrt[3]{6n}\rfloor> \begin{binom}{\lfloor \sqrt[3]{6n}\rfloor+2}{3}\end{binom}\\
n+\lfloor \sqrt[3]{6n}\rfloor-1 & \mbox{otherwise.}
\end{array}
\right.
\end{equation*}
This formula requires only one evaluation of the integer cube root function. This is simpler and more amenable to algorithmic implementation than the formula in \cite{mortici2010} which was obtained by solving a cubic polynomial and requires two evaluations of the (real-valued) cube root function and one evaluation of the square root function.

Similar to Section \ref{sec:pyramidal}, for values of $k$ other than those indicated in Theorem \ref{thm:simplex1}, Equation (\ref{eqn:rf}) still holds, albeit only for $n\geq n_0(k)$ for some value $n_0(k)>0$.
On the other hand, by looking at all $n\leq n_0(k)$, we find that Equation (\ref{eqn:simplex}) holds for all $n\geq 1$ and $1<k\leq 15$.

\begin{corollary}\label{cor:simplex2}
For $1<k\leq 15$, the $n$-th non-$k$-simplex number is given by Equation (\ref{eqn:simplex}).
\end{corollary}

\subsubsection{Complement to Sequences Generated by Polynomials}
Consider an integer sequence defined by a $k$-th degree polynomial as follows: $a(n) = \left\lfloor\sum_{i=0}^k a_in^i\right\rfloor$ with $k>1$, real coefficients $a_i$, and $a_k>0$ such that for all $n\in \mathbb{N}^+$, we have $a(n+1)>a(n)\geq 0$. 
For instance, the sequences defined by $f_{\neg P}(n)$ in Sections \ref{sec:polygonal}-\ref{sec:simplex} are all of this form.
\begin{lemma}
For each real number $b$ such that $\frac{a_{k-1}}{a_k}-1<b<\frac{a_{k-1}}{a_k}$,
$$ a_k(n+b)^k \leq a(n) \leq a_k(n+1+b)^k$$
for all sufficiently large $n$.
If in addition $k\geq 3$, then
$$ a_k(n+b)^k +n \leq a(n) \leq a_k(n+1+b)^k$$
for all sufficiently large $n$.
\end{lemma}
\begin{proof}
Let $\tilde{a}(n) =  \sum_{i=0}^k a_in^i$, then $|a(n)-\tilde{a}(n)|\leq 1$.
Next note that $a_k(n+b)^k -\tilde{a}(n) = (a_kb-a_{k-1})n^{k-1}+r_1(n)$ where $r_1(n)$ is a polynomial of degree $k-2$ or less. Since $a_kb-a_{k-1}<0$, it follows that $a_k(n+b)^k -\tilde{a}(n)<-1$ and thus $a_k(n+b)^k -a(n)<0$ for sufficiently large $n$. Similarly,
$a_k(n+b+1)^k-\tilde{a}(n) =  (a_k(b+1)-a_{k-1})n^{k-1} + r_2(n)$ for a polynomial $r_2(n)$ of degree $k-2$ or less.
Since $a_k(b+1)-a_{k-1} > 0$, this means that $a_k(n+b+1)^k-\tilde{a}(n)>1$ and thus $a_k(n+b+1)^k-a(n)>0$ for large enough $n$.
Finally, if $k\geq 3$, the facts that $a_k(n+b)^k +n-\tilde{a}(n) = (a_kb-a_{k-1})n^{k-1}+n+r_1(n)$, $(a_kb-a_{k-1}) < 0$, and $n+r_1(n)$ is a polynomial of degree $k-2$ or less imply that $a_k(n+b)^k +n-\tilde{a}(n)<-1$ and thus $a_k(n+b)^k +n-a(n)<0$ for large enough $n$.
\end{proof}

Note that $a(n+1)-a(n)\geq n$ for large enough $n$ since $a_k>0$.
By choosing $\alpha(m) =  a_k(m+b)^k$ with corresponding $h(n) = \left\lfloor\sqrt[k]{\frac{n}{a_k}}-b\right\rfloor$, for large enough $n$ the $n$-th term of the complementary sequence to $a(n)$ can be found using one evaluation of the function $h$. In particular,
Theorem \ref{thm:interleave} implies the following result.

\begin{theorem}
Let $\{c(n)\}$ be the complementary sequence to the sequence $\{a(n)\}$. If $\frac{a_{k-1}}{a_k}-1<b<\frac{a_{k-1}}{a_k}$ and $h(n) = \left\lfloor\sqrt[k]{\frac{n}{a_k}}-b\right\rfloor$, then there exists $n_0>0$ such that for all $n\geq n_0$,
\begin{equation*} 
c(n) = \left\{ \begin{array}{ll} n+h(n)+1 & \text{if } \quad n+h(n)\geq a(h(n)+1)\\
n+h(n)-1 & \text{if } \quad n+h(n)\leq a(h(n))\\
n+h(n) & \text{otherwise.} 
\end{array} 
\right.
\label{eqn:polynomial_seq}
\end{equation*}
\label{thm:polynomial_seq}
\end{theorem} 

This can be implemented  for the following special case using the integer $k$-th root function $n\rightarrow\lfloor\sqrt[k]{n}\rfloor$ discussed in Section \ref{sec:intro} by choosing $b=\left\lfloor\frac{a_{k-1}}{a_k}\right\rfloor$.
\begin{corollary}
Let $\{c(n)\}$ be the complementary sequence to the sequence $\{a(n)\}$. If $\frac{a_{k-1}}{a_k}$ is not an integer, then there exists $n_0>0$ such that for all $n\geq n_0$,
\begin{equation*} 
c(n) = \left\{ \begin{array}{ll} n+\left\lfloor\sqrt[k]{\frac{n}{a_k}}\right\rfloor-\left\lfloor\frac{a_{k-1}}{a_k}\right\rfloor+1 & \text{if }  \quad n+\left\lfloor\sqrt[k]{\frac{n}{a_k}}\right\rfloor-\left\lfloor\frac{a_{k-1}}{a_k}\right\rfloor\geq \\[5pt] &\quad\quad a\left(\left\lfloor\sqrt[k]{\frac{n}{a_k}}\right\rfloor-\left\lfloor\frac{a_{k-1}}{a_k}\right\rfloor+1\right)\\[5pt]
n+\left\lfloor\sqrt[k]{\frac{n}{a_k}}\right\rfloor-\left\lfloor\frac{a_{k-1}}{a_k}\right\rfloor-1 & \text{if } \quad n+\left\lfloor\sqrt[k]{\frac{n}{a_k}}\right\rfloor-\left\lfloor\frac{a_{k-1}}{a_k}\right\rfloor\leq\\[5pt]&\quad\quad a\left(\left\lfloor\sqrt[k]{\frac{n}{a_k}}\right\rfloor-\left\lfloor\frac{a_{k-1}}{a_k}\right\rfloor\right)\\
n+\left\lfloor\sqrt[k]{\frac{n}{a_k}}\right\rfloor-\left\lfloor\frac{a_{k-1}}{a_k}\right\rfloor &\text{otherwise.} 
\end{array} 
\right.
\end{equation*}
\label{cor:polynomial_seq_1}
\end{corollary}

Similarly, Corollary \ref{cor:interleave2b} implies that:

\begin{corollary}
Let $k\geq 3$ and let $\{c(n)\}$ be the complementary sequence to the sequence $\{a(n)\}$.  If $\frac{a_{k-1}}{a_k}-1<b<\frac{a_{k-1}}{a_k}$ and $h(n) = \left\lfloor\sqrt[k]{\frac{n}{a_k}}-b\right\rfloor$, then there exists $n_0>0$ such that for all $n\geq n_0$,
\begin{equation*} c(n) = \left\{ \begin{array}{ll} n+h(n) & \mbox{if } \quad n+h(n)> a(h(n))\\
 n+h(n)-1 & \mbox{otherwise.}
\end{array} 
\right.
\label{eqn:polynomial_2b}
\end{equation*}
\label{cor:polynomial_2b}
\end{corollary}

\begin{corollary}
Let $k\geq 3$ and let $\{c(n)\}$ be the complementary sequence to the sequence $\{a(n)\}$.  If $\frac{a_{k-1}}{a_k}$ is not an integer, then there exists $n_0>0$ such that for all $n\geq n_0$,
\begin{equation*} c(n) = \left\{ \begin{array}{ll} n+\left\lfloor\sqrt[k]{\frac{n}{a_k}}\right\rfloor-\left\lfloor\frac{a_{k-1}}{a_k}\right\rfloor & \mbox{if }  \quad n+\left\lfloor\sqrt[k]{\frac{n}{a_k}}\right\rfloor-\left\lfloor\frac{a_{k-1}}{a_k}\right\rfloor>\\[5pt]&\quad\quad a\left(\left\lfloor\sqrt[k]{\frac{n}{a_k}}\right\rfloor-\left\lfloor\frac{a_{k-1}}{a_k}\right\rfloor\right)\\
 n+\left\lfloor\sqrt[k]{\frac{n}{a_k}}\right\rfloor-\left\lfloor\frac{a_{k-1}}{a_k}\right\rfloor-1 & \mbox{otherwise.}
\end{array} 
\right.
\label{eqn:polynomial_2c_1}
\end{equation*}
\label{cor:polynomial_2c_1}
\end{corollary}

\subsubsection{Computing Characteristic Functions}
If $h$ is easily computable, then this can lead to an efficient algorithm to compute the characteristic function $\chi_{_{\neg P}}$ of $f_{\neg P}$. It is clear that if $f_{\neg P}(m) = n$, then $h(n) = m$.
This implies that $\chi_{_{\neg P}}(n) = 1$ if and only if $f_{\neg P}(h(n)) = n$, i.e., $$\chi_{_{\neg P}}(n) = \llbracket f_{\neg P}(h(n)) = n\rrbracket.$$
As an example, consider the characteristic function $\chi(n)$ of $4$-simplex numbers, i.e., numbers of the form $\binom{m}{4}$ for some $m$ (OEIS \href{https://oeis.org/A256436}{A256436}).
Using $h(n) = \lfloor\sqrt[4]{24n}\rfloor-1$, we see that $\chi(n) = 1$ if and only if $n=\binom{\lfloor\sqrt[4]{24n}\rfloor+2}{4}$.

\subsection{Bisection Search}
Theorem \ref{thm:twostep} shows that if $f_{\neg P}(n)$ grows quadratically (or faster), then the number of steps needed to find $f_P$ using the function iteration method is no more than two. 
Numerical experiments suggest a similar relationship for other grow rates. In particular, these experiments allow us to conjecture the following result.
\begin{conjecture}
For $k\geq 1$, if $f_{\neg P}(n)$ grows faster than $n^{1+\frac{1}{k}}$, then the number of steps needed for the function iteration method to determine $f_{P}(n)$ is no more than $k+1$.
\end{conjecture}
  
Thus, for the complementary sequences discussed in the above sections, the number of steps is bounded by a constant independent of $n$.
For other types of sequences, this may not be the case.
In these cases, a different method for finding fixed points is needed that is more efficient than the function iteration method.

The function iteration method can take a large number of steps to converge when the set of integers that satisfy $P$ is sparse. In this case, it might be more optimal to use a bisection search method to find the fixed point of $g_n$. To this end, we first find an interval $[k_{\min},k_{\max}]$ that bounds $f_P(n)$.
Since we know that $f_{P}(n)\geq n$, we initially set $k_{\min}=n$. If a better lower bound for $f_P(n)$ is known, this can be assigned to the initial $k_{\min}$.  Initially we can also set $k_{\max} = n$ unless we know a better lower or upper bound or an approximation of $b$ for $f_P(n)$ in which case we set $k_{\max} = b$. 

We next double this initial value $k_{\max}$ repeatedly until $g_n(k_{\max})\leq k_{\max}$. Then a bisection search is applied until
the smallest fixed point is obtained. The pseudo code for this algorithm is shown in Algorithm \ref{alg:bisectionCP}.
The number of steps needed to converge is on the order of $\log_2(f_P(n))$ and is in general more efficient than the function iteration method in Section \ref{sec:iteration}, especially when the numbers satisfying $P$ are sparse. 
To illustrate this, we show in Figure \ref{fig:sixprimes} the number of steps needed for these two methods to obtain $f_P(n)$ when $P(m)$ denotes the logical statement ``$m$ is a product of exactly $6$ distinct primes'' (OEIS \href{https://oeis.org/A067885}{A067885}). We see that 
the number of steps for the bisection method is much less than for the function iteration method.

\begin{algorithm}[htbp]
\begin{algorithmic}
\Require $r\in\mathbb{N}, g_n(x)$
\Comment{computes the minimal fixed point of $g_n$.}
\State $k_{\min} \gets n$, $k_{\max}\gets n$.
\While {$g_n(k_{\max}) > k_{\max}$}
\State $k_{\max} \gets 2k_{\max}$ 
\EndWhile
\State $k_{\min} \gets \max(k_{\min},k_{\max}/2)$
\While {$k_{\max}-k_{\min} > 1$}
\State $k_{\mbox{\footnotesize mid}} = \lfloor (k_{\max}+k_{\min})/2\rfloor$
\If {$g(k_{\mbox{\footnotesize mid}})\leq k_{\mbox{\footnotesize mid}}$}
\State $k_{\max} \gets k_{\mbox{\footnotesize mid}}$
\Else
\State $k_{\min} \gets k_{\mbox{\footnotesize mid}}$
\EndIf
\EndWhile
\State \Return $k_{\max}$
\end{algorithmic}
\caption{Bisection search on $g_n(x)$ to compute $f_P(n)$.}
\label{alg:bisectionCP}
\end{algorithm}

\begin{figure}[htbp]
\begin{center}
\includegraphics[width=6in]{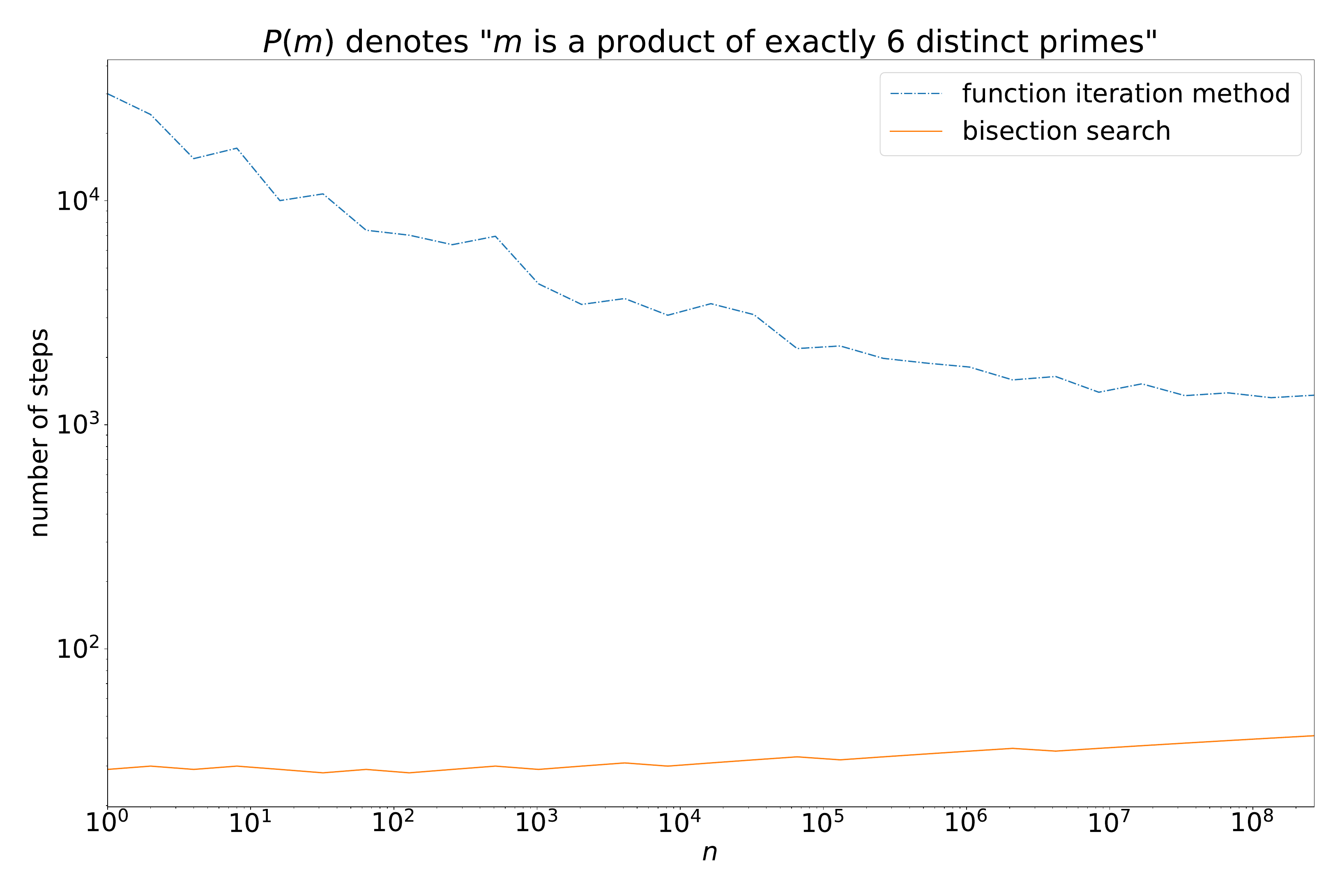}
\end{center}
\caption{Number of steps to find $f_P(n)$ when  $P(m)$ denotes the statement ``$m$ is a product of exactly $6$ distinct primes''.}
\label{fig:sixprimes}
\end{figure}

\subsection{Hybrid Method}
Since in several well known cases the number of steps that the function iteration method terminates in is small, we can take advantage of this by setting the initial $k_{\min}$ and $k_{\max}$ to $g_n^{(r)}(n)$ for some small $r$, say $r=2$. Here $f^{(r)}$ denote the $r$-th iterate of the function $f$. The total number of steps is then $r$ plus the number of steps of the bisection method. This is illustrated in Algorithm \ref{alg:hybridCP}.

\begin{algorithm}[htbp]
\begin{algorithmic}
\Require $r\in\mathbb{N}, g_n(x)$
\Comment{computes the minimal fixed point of $g_n$.}
\State $k_{\min} \gets g_n^{(r)}(n)$, $k_{\max}\gets g_n^{(r)}(n)$.
\While {$g_n(k_{\max}) > k_{\max}$}
\State $k_{\max} \gets 2k_{\max}$ 
\EndWhile
\State $k_{\min} \gets \max(k_{\min},k_{\max}/2)$
\While {$k_{\max}-k_{\min} > 1$}
\State $k_{\mbox{\footnotesize mid}} = \lfloor (k_{\max}+k_{\min})/2\rfloor$
\If {$g(k_{\mbox{\footnotesize mid}})\leq k_{\mbox{\footnotesize mid}}$}
\State $k_{\max} \gets k_{\mbox{\footnotesize mid}}$
\Else
\State $k_{\min} \gets k_{\mbox{\footnotesize mid}}$
\EndIf
\EndWhile
\State \Return $k_{\max}$
\end{algorithmic}
\caption{Hybrid method. First $r$ iterations of the function iteration method is used to initialize the bisection search.}
\label{alg:hybridCP}
\end{algorithm}

\section{The $n$-th Term of the Union of Two Sequences}
Consider the following scenario where $\{a(i)\}$ and $\{b(i)\}$ are disjoint sequences of integers with corresponding counting functions $C_{a}$ and $C_b$. For instance, the sequence $a$ could be the set of square-free numbers and $b$ the perfect powers. The goal is to find the $n$-th element in the sorted list when $a$ and $b$ are sorted together. In this specific example of $a$ and $b$, the joint sequence (denoted as $c$) is OEIS \href{https://oeis.org/A304449}{A304449}. Other examples of such joint sequences are for instance OEIS \href{https://oeis.org/A000430}{A000430}, \href{https://oeis.org/A006899}{A006899}, \href{https://oeis.org/A089237}{A089237}, \href{https://oeis.org/A126684}{A126684}, \href{https://oeis.org/A168363}{A168363} and \href{https://oeis.org/A174090}{A174090}.
Because the sequences $a$ and $b$ are disjoint, in this scenario, the counting function $C_P(n)$ for the joint sequence $c$ is simply the sum of the counting functions of $a$ and $b$ given by $C_a(n)+C_b(n)$ and the above algorithms can be used to find the $n$-th element of $c$.

When the two sequences are not disjoint, $C_P(n) = C_a(n)+C_b(n)-C_{a\cap b}(n)$ by the inclusion-exclusion principle and in some cases the intersection of the sequences can easily be determined.
For instance, let $p$ be prime and consider the sequence of numbers $k$ such that $k^k$ is a $p$-th power (OEIS \href{https://oeis.org/A176693}{A176693},  \href{https://oeis.org/A376279}{A376379}). If the prime factorization of $k$ is $k=\prod_i p_i^{e_i}$, then $k^k=\prod_i p_i^{ke_i}$. Thus, $k^k$ is a $p$-power if and only if $ke_i \equiv 0  \pmod{p}$. Since $p$ is prime, the residue classes form a field, and this condition corresponds to when $k$ is a multiple of $p$ or $e_i$ is a multiple of $p$ for all $i$, i.e., $k$ is a $p$-th power. Thus, this sequence is the union of the multiples of $p$ and the $p$-th powers with corresponding counting functions
$\lfloor n/p\rfloor$ and $\lfloor\sqrt[p]{n}\rfloor$. The counting function of their intersection is
$\lfloor\sqrt[p]{n}/p\rfloor$ and thus the counting function of the union is $\lfloor n/p\rfloor+\lfloor\sqrt[p]{n}\rfloor - \lfloor\sqrt[p]{n}/p\rfloor$.

Similarly the counting function of the union of squares and powers of $2$ (OEIS \href{https://oeis.org/A188915}{A188915}) is given by $\lfloor \sqrt{n}\rfloor+ \lfloor \log_2(n)\rfloor-\lfloor \log_2(n)/2\rfloor=\lfloor \sqrt{n}\rfloor +\lceil \log_2(n)/2\rceil$. 
Similarly, the counting function of the sequence resulting from combining $a$ and $b$ and removing their intersection is
 $C_P(n) = C_a(n)+C_b(n)-2C_{a\cap b}(n)$. For an example see OEIS \href{https://oeis.org/A377025}{A377025}.

In computer science, there is sometimes the need to find the $n$-th smallest element of an (large) unordered list of length $l$ without having to sort the entire list.
Typically, this is performed using a partial sort (see e.g., the Quickselect algorithm \cite{Hoare1961}) which has an $O(l)$ average performance.
The above scenario can be considered a special case of this problem where the sequence can be decomposed into $k$ disjoint subsequences,  the counting functions of the subsequences are efficiently computed, and the length of the list of elements is a priori unknown.
Using the above algorithm leads to a running time depending on $n$ unlike the partial sort algorithm which has a running time depending on the length $l$ of the entire list.

\section{Sequences of Repeated Terms}
Since $f_P(n)$, which is the complementary sequence of the sequence $f_{\neg P}(n)$, can be viewed as the sequence of integers skipping the values of $f_{\neg P}(n)$,
we can consider the function $f_{P}(n)-n$ which list consecutive integers, each one of which is repeated. For example, consider the sequence of non-square numbers (OEIS \href{https://oeis.org/A000037}{A000037}): $f_{P}(n) = (2,3,5,6,7,8,10,11,\ldots)$.
The sequence $f_{P}(n)-n$ is $(1,1,2,2,2,2,3,3,\ldots)$, i.e., each integer $m$ appears $2m$ times in the sequence (OEIS \href{https://oeis.org/A000194}{A000194}).
Since $f_{P}(n) = n+\lfloor\sqrt{n}+\frac{1}{2}\rfloor$, this implies that the sequence  $(1,1,2,2,2,2,3,3,\ldots)$ can be written as $\lfloor\sqrt{n}+\frac{1}{2}\rfloor$.
Similarly, the sequence $f_{P}(n)$ of the non-triangular numbers corresponds to a sequence $f_P(n)-n$ where each integer $m$ appears $m$ times and thus can be written explicitly as $\lfloor\sqrt{2n}+\frac{1}{2}\rfloor$ (OEIS \href{https://oeis.org/A002024}{A002024}). This sequence has been studied in \cite{knuth:taocp:1968}.

More generally, given a sequence of real numbers $a_1, a_2,\ldots $, consider a sequence $b(n)$ for $n\geq 1$ where each number $a_m\geq 1$ appears $\beta(m)$ times consecutively:
$$ \underbrace{a_1,a_1,\ldots a_1}_{\beta(1) \text{ times}}, \underbrace{a_2,a_2,\ldots a_2}_{\beta(2) \text{ times}},\ldots$$
The goal is to determine $b(n)$ given $n$. The case of $\beta(m) = md$ for a fixed $d$ was studied in \cite{Nyblom2002}. 

We will consider the special case where $a_i = i$ as the approach to the general case is the same.
Then $f_{P}(n)=b(n)+n-1$ skips an integer at every additional $\beta(i)$ numbers, meaning the $i$-th value skipped is $\beta(i)+1$ plus the last number skipped. In other words
$f_{\neg P}(n) = \sum_{i=1}^n(\beta(i)+1) = n+\sum_{i=1}^n\beta(i)$.
We can then apply the results and algorithms in the previous sections to find $f_{P}(n)$ and thus also $b(n)$. 

Since $f_{\neg P}(n)-f_{\neg P}(n-1) = \beta(n)+1$, if in addition $\beta(m)\geq m-2$, then we can apply Theorems \ref{thm:twostep} and \ref{thm:onestep} to find $b(n)$.
To illustrate this approach, consider the sequence $b(n)$ where each integer $m$ appears $m^2$ times (OEIS \href{https://oeis.org/A074279}{A074279}). 
Then $f_{\neg P}(n) = n+\sum_{i=1}^n i^2 = n+n(n+1)(2n+1)/6$ (OEIS \href{https://oeis.org/A145066}{A145066}). Theorem \ref{thm:faulhaber_sum_2} in Section \ref{sec:sum_of_powers} shows that 
\begin{equation*}
f_{P}(n) = 
\left\{
\begin{array}{ll}
 n+ \lfloor\sqrt[3]{3n}\rfloor & \mbox{if }\quad 6n> \lfloor\sqrt[3]{3n}\rfloor \left( \lfloor\sqrt[3]{3n}\rfloor+1\right)\left(2 \lfloor\sqrt[3]{3n}\rfloor+1\right)\\
 n+ \lfloor\sqrt[3]{3n}\rfloor-1 & \mbox{otherwise.}
 \end{array}
 \right.
 \end{equation*}
This implies that 
\begin{equation*}
b(n) = f_P(n)-n+1 =
\left\{
\begin{array}{ll}
 \lfloor\sqrt[3]{3n}\rfloor +1 & \mbox{if }\quad 6n> \lfloor\sqrt[3]{3n}\rfloor \left( \lfloor\sqrt[3]{3n}\rfloor+1\right)\left(2 \lfloor\sqrt[3]{3n}\rfloor+1\right) \\
\lfloor\sqrt[3]{3n}\rfloor & \mbox{otherwise.}
  \end{array}
 \right.
 \end{equation*}

This formula is simpler than the formula for this sequence given in \cite{putievskiy2023integersequencesirregulararrays}.
More generally, a sequence $b(n)$ where each integer $m$ repeats $m^{k-1}$ times, can be computed using the formula

\begin{equation*}
b(n) = 
\left\{
\begin{array}{ll}
 \lfloor\sqrt[k]{kn}\rfloor +1 & \mbox{if }\quad kn>\sum_{j=0}^{k-1}\binom{k}{j}B^+_j \lfloor\sqrt[k]{kn}\rfloor^{k-j}\\
 \lfloor\sqrt[k]{kn}\rfloor & \mbox{otherwise.}
  \end{array}
 \right.
 \end{equation*} 
Again, these formulas require only one evaluation of $\lfloor\sqrt[k]{kn}\rfloor$. 

This approach is used to find novel formulas for sequences of repeated integers such as OEIS \href{https://oeis.org/A056556}{A056556},    \href{https://oeis.org/A056557}{A056557}, \href{https://oeis.org/A056558}{A056558}, \href{https://oeis.org/A108581}{A108581}, \href{https://oeis.org/A108582}{A108582}, \href{https://oeis.org/A127321}{A127321}, \href{https://oeis.org/A180447}{A180447}, \href{https://oeis.org/A194847}{A194847}, \href{https://oeis.org/A194848}{A194848}, \href{https://oeis.org/A235463}{A235463},  and \href{https://oeis.org/A360010}{A360010}.
For instance, the sequence $b(n)$ where each integer $m$ is repeated $\begin{binom}{m+3}{3}\end{binom}$ times (OEIS \href{https://oeis.org/A127321}{A127321}) can be expressed as 

\begin{equation*}
b(n) = 
\left\{
\begin{array}{ll}
\lfloor\sqrt[4]{24(n+2)}\rfloor-2 & \mbox{if }\quad n<\begin{binom}{\lfloor\sqrt[4]{24(n+2)}\rfloor+2}{4}\end{binom}\\
\lfloor\sqrt[4]{24(n+2)}\rfloor-1 & \mbox{otherwise.}
  \end{array}
 \right.
 \end{equation*} 

As mentioned above, the existence of a simple formula for $f_P$ does not imply a simple formula for $f_{\neg P}$. However, the existence of a simple formula for $C_P$ implies the existence of a simple formula $C_{\neg P}$ which leads to efficient algorithms for both $f_P$ and $f_{\neg P}$. In Section \ref{sec:cp}, we look at other examples of logical statements $P$, mainly related to the prime factorizations of numbers, for which there are relatively efficient algorithms for computing $C_P(n)$ (and thus also $C_{\neg P}(n)$).
	
\section{Some Other Explicitly Computable Counting Functions $C_P(n)$} \label{sec:cp}
The algorithm for computing the complementary sequence $f_{\neg P}(n)$ requires the computation of the counting function $C_P(n)$. Note that $C_P(n)$
is the largest integer $m$ such that $f_P(m)\leq n$. If $f_P(n)$ is efficiently computed, the fact that $f_P$ is strictly increasing means that $C_P$ can be computed via a bisection search (similar to  Algorithm \ref{alg:bisectionCP} for $g_n$). An example of such functions $f_P$ is the case where $P(m)$ denotes the statement ``$q(x)=m$ for some positive integer $x$'', with $q$ a strictly increasing polynomial with integer coefficients.

In other cases, the function $f_P(n)$ is more easily enumerated sequentially in which case a naive approach would be to enumerate $f_P(n)$ in order to compute $C_P(n)$ and then solve for the fixed point of $g_n$. However, this approach to find the complementary sequence $f_{\neg P}(n)$ is less efficient than simply enumerating $f_P(n)$ and assigning the gaps between successive terms to $f_{\neg P}(n)$.

On the other hand, for several number theoretical statements $P$, computing  $C_P(n)$ can be more efficient than enumerating $f_P(n)$.
For instance, there exists algorithms for computing the prime counting function $\pi(x)$
that are more efficient than enumerating all $\pi(x)$ prime numbers less than or equal to $x$ \cite{Deleglise1996}.  

The formulas for the counting function of square-free numbers\footnote{See \cite{pawlewicz2011countingsquarefreenumbers} for a more efficient formula.} given by 
$$\sum_{i=1}^{\lfloor\sqrt{n}\rfloor} \mu(i)\lfloor n/i^2\rfloor$$ and the counting function of perfect powers\footnote{See \cite{nyblom2006} for an alternative formula.} given by 
$$1-\sum_{i=2}^{\lfloor \log_2(x)\rfloor}\mu(i)(\lfloor \sqrt[i]{x}\rfloor -1)$$ require the computation of  the M\"{o}bius function $\mu(n)$ which can be done efficiently using a sieve \cite{Deleglise1996a}.  The counting functions of semiprimes and $k$-almost primes can be expressed using $\pi(x)$ and are found in  \cite{almostprime,crisan:semiprimes:2021}. The counting function of square-free $k$-almost primes (i.e., numbers that are products of $k$ distinct primes) has a similar form in terms of $\pi(x)$ (see for instance OEIS \href{https://oeis.org/A067885}{A067885}).

Another example where $C_P(n)$ is easily obtained is for instance when $P(m)$ denotes the statement  ``$m$ is a repdigit in base $b$'' (OEIS \href{https://oeis.org/A139819}{A139819}). In this case,
$C_P(n) = (b-1)\lfloor\log_b(n)\rfloor+\left\lfloor\frac{(b-1)n}{b^{\lfloor\log_b(n)\rfloor+1}-1}\right\rfloor$.

\section{Conclusions}
We study algorithms to find the $n$-th integer that satisfies a certain condition $P$ via a fixed point approach. We show that the function iteration method to solve this problem is equivalent to the Lambek-Moser method. Furthermore, we show that the two-step iteration of the Lambek-Moser algorithm requiring $2$ evaluations of the counting function can be reduced to a single evaluation of the counting function. We also present a bisection algorithm that is more efficient when the numbers satisfying $P$ are sparse. 

For a particular condition $P$, this approach is useful not only in finding the $n$-th term of the  complementary sequences $f_{\neg P}$ but also the $n$-th term of the original sequence $f_P$ itself when computing $C_P(n)$ is more efficient than enumerating $f_P$ up to $n$. For instance, we have used
this approach in various Python programs to find the $n$-th term of
OEIS sequences without enumerating all $n$ terms. 
These sequences include perfect powers (OEIS \href{https://oeis.org/A001597}{A001597}), prime powers (OEIS \href{https://oeis.org/A000961}{A000961}, \href{https://oeis.org/A246655}{A246655}, \href{https://oeis.org/A246547}{A246547}, \href{https://oeis.org/A025475}{A025475}), powerful numbers (OEIS \href{https://oeis.org/A001694}{A001694}), $k$-full numbers (OEIS \href{https://oeis.org/A036966}{A036966}, \href{https://oeis.org/A036967}{A036967}, \href{https://oeis.org/A069492}{A069492}, \href{https://oeis.org/A069493}{A069493}), square-free numbers (OEIS \href{https://oeis.org/A005117}{A005117}), semiprimes (OEIS \href{https://oeis.org/A001358}{A001358}), square-free semiprimes (OEIS \href{https://oeis.org/A006881}{A006881}), $k$-almost primes (OEIS \href{https://oeis.org/A014612}{A014612}, \href{https://oeis.org/A014613}{A014613}, \href{https://oeis.org/A014614}{A014614}), square-free $k$-almost primes (OEIS \href{https://oeis.org/A007304}{A007304}, \href{https://oeis.org/A046386}{A046386}, \href{https://oeis.org/A046387}{A046387}, \href{https://oeis.org/A067885}{A067885}, \href{https://oeis.org/A281222}{A281222}), Achilles numbers (OEIS \href{https://oeis.org/A052486}{A052486}), 
orders of proper semifields and twisted fields (OEIS \href{https://oeis.org/A088247}{A088247}, \href{https://oeis.org/A088248}{A088248}), $p$-smooth numbers (OEIS \href{https://oeis.org/A003586}{A003586}, \href{https://oeis.org/A051037}{A051037}, \href{https://oeis.org/A002473}{A002473}, \href{https://oeis.org/A051038}{A051038}, \href{https://oeis.org/A080197}{A080197}),
numbers with at least one digit $b-1$ in base $b$ (OEIS \href{https://oeis.org/A074940}{A074940},  \href{https://oeis.org/A337239}{A337239}, \href{https://oeis.org/A337250}{A337250}), primes starting with digit $b$ (OEIS \href{https://oeis.org/A045707}{A045707}-\href{https://oeis.org/A045715}{A045715}),
sums of $3$ squares (OEIS \href{https://oeis.org/A000378}{A000378}),
numbers with exactly $k$ divisors (OEIS \href{https://oeis.org/A030513}{A030513}, \href{https://oeis.org/A030515}{A030515}, \href{https://oeis.org/A030626}{A030626}, \href{https://oeis.org/A030627}{A030627}, \href{https://oeis.org/A030632}{A030632},  \href{https://oeis.org/A030633}{A030633}, \href{https://oeis.org/A137484}{A137484}, \href{https://oeis.org/A137485}{A137485}, \href{https://oeis.org/A137488}{A137488}), selected sifting sequences \cite{Snellman2025} (OEIS \href{https://oeis.org/A003159}{A003159}, \href{https://oeis.org/A007417}{A007417}, \href{https://oeis.org/A382744}{A382744}, \href{https://oeis.org/A382745}{A382745}, \href{https://oeis.org/A382746}{A382746}),
characteristic functions (OEIS \href{https://oeis.org/A256436}{A256436}, \href{https://oeis.org/A387646}{A387646}),
and complementary sequences (OEIS \href{https://oeis.org/A004215}{A004215}, \href{https://oeis.org/A024619}{A024619}, \href{https://oeis.org/A052485}{A052485}, \href{https://oeis.org/A029742}{A029742}, \href{https://oeis.org/A007916}{A007916}, \href{https://oeis.org/A002808}{A002808}, \href{https://oeis.org/A013929}{A013929},  \href{https://oeis.org/A100959}{A100959}, \href{https://oeis.org/A139819}{A139819}, \href{https://oeis.org/A374812}{A374812}, \href{https://oeis.org/A090946}{A090946}, \href{https://oeis.org/A057854}{A057854}, \href{https://oeis.org/A185543}{A185543}, \href{https://oeis.org/A138836}{A138836}, \href{https://oeis.org/A138890}{A138890}, \href{https://oeis.org/A325112}{A325112}, \href{https://oeis.org/A059485}{A059485}, \href{https://oeis.org/A279622}{A279622}, \href{https://oeis.org/A145397}{A145397}, \href{https://oeis.org/A302058}{A302058}, \href{https://oeis.org/A376573}{A376573}, \href{https://oeis.org/A183300}{A183300}, \href{https://oeis.org/A387644}{A387644}).

Finally, interested readers can obtain Python programs for the computation of the OEIS sequences discussed in this paper by accessing the entries of the corresponding OEIS sequences or the corresponding programs in the GitHub repository \verb+oeis-sequences+ \cite{Wu:github:oeissequences:2021}.


\begin{thebibliography}{10}

\bibitem{Lambek1954}
J.~Lambek and L.~Moser, ``Inverse and complementary sequences of natural
  numbers,'' {\em The American Mathematical Monthly}, vol.~61, no.~7, p.~454,
  1954.

\bibitem{Honsberger1970}
R.~Honsberger, {\em Ingenuity in Mathematics}, ch.~Essay 12, pp.~93--110.
\newblock The Mathematical Association of America, 1970.

\bibitem{Nelson1988}
R.~D. Nelson, ``Sequences which omit powers,'' {\em The Mathematical Gazette},
  vol.~72, no.~461, pp.~208--211, 1988.

\bibitem{mortici2010}
C.~Mortici, ``Remarks on complementary sequences,'' {\em Fibonacci Quarterly},
  vol.~48, no.~4, pp.~343--347, 2010.

\bibitem{Reis1990}
A.~J.~D. Reis and D.~M. Silberger, ``Generating nonpowers by formula,'' {\em
  Mathematics Magazine}, vol.~63, no.~1, pp.~53--55, 1990.

\bibitem{Nyblom2002}
M.~A. Nyblom, ``Some curious sequences involving floor and ceiling functions,''
  {\em The American Mathematical Monthly}, vol.~109, no.~6, p.~559, 2002.

\bibitem{graham:concrete_math:1994}
R.~L. Graham, D.~E. Knuth, and O.~Patashnik, {\em Concrete Mathematics: A
  Foundation for Computer Science}.
\newblock Addison-Wesley, 2nd~ed., 1994.

\bibitem{gould1965}
H.~W. Gould, ``Non {F}ibonacci numbers,'' {\em Fibonacci Quarterly}, vol.~3,
  pp.~177--183, 1965.

\bibitem{Farhi2011}
B.~Farhi, ``An explicit formula generating the non-{F}ibonacci numbers.''
  arXiv:1105.1127 [math.NT], May 2011.

\bibitem{Heath1921}
T.~Heath, {\em A History of Greek Mathematics}, vol.~2.
\newblock Clarendon Press, 1921.

\bibitem{iverson1962}
K.~E. Iverson, {\em A Programming Language}.
\newblock Wiley, 1962.

\bibitem{knuth:notation:1992}
D.~E. Knuth, ``Two notes on notation,'' {\em The American Mathematical
  Monthly}, vol.~99, no.~5, pp.~403--422, 1992.

\bibitem{oeis}
{The OEIS Foundation Inc.}, ``The on-line encyclopedia of integer sequences,''
  1996-present.
\newblock Founded in 1964 by N. J. A. Sloane.

\bibitem{MacMillan2011}
K.~MacMillan and J.~Sondow, ``Proofs of power sum and binomial coefficient
  congruences via {P}ascal’s identity,'' {\em The American Mathematical
  Monthly}, vol.~118, no.~6, p.~549, 2011.

\bibitem{Hoare1961}
C.~A.~R. Hoare, ``Algorithm 65: find,'' {\em Communications of the ACM},
  vol.~4, no.~7, pp.~321--322, 1961.

\bibitem{knuth:taocp:1968}
D.~E. Knuth, {\em The Art of Computer Programming}.
\newblock Addison-Wesley, 1968.

\bibitem{putievskiy2023integersequencesirregulararrays}
B.~Putievskiy, ``Integer sequences: Irregular arrays and intra-block
  permutations.'' arXiv:2310.18466 [math.CO], 2023.

\bibitem{Deleglise1996}
M.~Deléglise and J.~Rivat, ``Computing $\pi(x)$: the {M}eissel, {L}ehmer,
  {L}agarias, {M}iller, {O}dlyzko method,'' {\em Mathematics of Computation},
  vol.~65, no.~213, pp.~235--245, 1996.

\bibitem{pawlewicz2011countingsquarefreenumbers}
J.~Pawlewicz, ``Counting square-free numbers.'' arXiv:1107.4890 [math.NT],
  2011.

\bibitem{nyblom2006}
M.~A. Nyblom, ``A counting function for the sequence of perfect powers,'' {\em
  Austral. Math. Soc. Gaz.}, vol.~33, pp.~338--343, 2006.

\bibitem{Deleglise1996a}
M.~Deléglise and J.~Rivat, ``Computing the summation of the {M}\"{o}bius
  function,'' {\em Experimental Mathematics}, vol.~5, no.~4, pp.~291--295,
  1996.

\bibitem{almostprime}
E.~W. Weisstein, ``Almost prime.'' From MathWorld--A Wolfram Web Resource.

\bibitem{crisan:semiprimes:2021}
D.~Cri\c{s}an and R.~Erban, ``On the counting function of semiprimes,'' {\em
  Integers}, vol.~21, p.~\#A122, 2021.

\bibitem{Snellman2025}
J.~Snellman, ``Greedy regular convolutions.'' arXiv:2504.02795 [math.NT], Apr.
  2025.

\bibitem{Wu:github:oeissequences:2021}
C.~W. Wu, ``oeis-sequences,'' GitHub Repository,
  \url{https://github.com/postvakje/oeis-sequences}, 2021-.

\end{thebibliography}
\end{document}